\documentclass{amsart}
\usepackage{hyperref}
\usepackage{tikz}
\usepackage{blkarray} 
\usepackage{multirow}

\hypersetup{
pdftitle={Graphs of Small Rank-width are Pivot-minors of Graphs of Small Tree-width}, 
pdfauthor={O-joung Kwon and Sang-il Oum}
}
\newcommand\abs[1]{\lvert #1\rvert}
\newtheorem{THM}{Theorem}[section]
\newtheorem{LEM}[THM]{Lemma}

\newtheorem{PROP}[THM]{Proposition}

\newcommand\rank{\operatorname{rank}}
\newcommand\rw{\operatorname{rw}}
\newcommand\lrw{\operatorname{lrw}}
\newcommand\tw{\operatorname{tw}}
\newcommand\pw{\operatorname{pw}}
\newcommand\cutrk{\operatorname{cutrk}}
\newcommand\pivot{\wedge}

\theoremstyle{definition}

\begin{document}
\title[Graphs of small rank-width are pivot-minors]{Graphs of Small Rank-width are Pivot-minors of Graphs of Small Tree-width}
\address{Department of Mathematical Sciences, KAIST, 291 Daehak-ro
  Yuseong-gu Daejeon, 305-701 South Korea}
\author{O-joung Kwon}
\email{ilkof@kaist.ac.kr}
\author{Sang-il Oum}
\email{sangil@kaist.edu}
\thanks{Supported by Basic Science Research
  Program through the National Research Foundation of Korea (NRF)
  funded by the Ministry of Education, Science and Technology
  (2012-0004119). S. O. is also supported by TJ Park Junior Faculty Fellowship.}
\date{\today}

\begin{abstract}
We prove that every graph of rank-width $k$ is a pivot-minor of a graph of tree-width at most $2k$. We also prove that graphs of rank-width at most $1$, equivalently distance-hereditary graphs, are exactly vertex-minors of trees, and graphs of linear rank-width at most $1$ are precisely vertex-minors of paths. 
In addition, we show that bipartite graphs of rank-width at most $1$ are exactly pivot-minors of trees and bipartite graphs of linear rank-width at most $1$ are precisely pivot-minors of paths. 
\end{abstract}

\maketitle

\section{Introduction} \label{sec:intro}
Rank-width is a width parameter of graphs, introduced by Oum and Seymour~\cite{Oum05},
measuring how easy it is to decompose a graph into a tree-like structure where the ``easiness'' is measured in terms of the matrix rank function derived from edges formed by vertex partitions. Rank-width is a generalization of another, more well-known width parameter called tree-width, introduced by Robertson and Seymour~\cite{RS1986a}. It is well known that every graph of small tree-width also has small rank-width; Oum~\cite{Oum08} showed that if a graph has tree-width $k$, then its rank-width is at most $k+1$. The converse does not hold in general, as complete graphs have rank-width $1$ and arbitrary large tree-width.

Pivot-minor and vertex-minor relations are graph containment relations such that rank-width cannot increase when taking pivot-minors or vertex-minors of a graph~\cite{Oum05}. Our main result is that for every graph $G$ with rank-width at most $k$ and $\abs{V(G)}\geq 3$, there exists a graph $H$ having $G$ as a pivot-minor such that $H$ has tree-width at most $2k$ and $\abs{V(H)}\leq (2k+1)\abs{V(G)}-6k$. Furthermore, we prove that for every graph $G$ with linear rank-width at most $k$ and $\abs{V(G)}\geq 3$, there exists a graph $H$ having $G$ as a pivot-minor such that $H$ has path-width at most $k+1$ and $\abs{V(H)}\leq (2k+1)\abs{V(G)}-6k$.

As a corollary, we give new characterizations of two graph classes:
graphs with rank-width at most $1$ and graphs with linear rank-width
at most $1$. We show that a graph has rank-width at most $1$ if and
only if it is a vertex-minor of a tree. We also prove that a graph has
linear rank-width at most $1$ if and only if it is a vertex-minor of a
path. Moreover, if the graph is bipartite, we prove that a
vertex-minor relation can be replaced with a pivot-minor relation in
both theorems. Table~\ref{table} summarizes our theorems.
 
\begin{table} \label{table}
\begin{center}
\begin{tabular}[t]{|c c c|}
\hline
$G$ has rank-width $\leq k$ & $\Rightarrow$ & $G$ is a pivot-minor of \\  
& & a graph of tree-width $\leq 2k$ \\
\hline
$G$ has linear rank-width $\leq k$ & $\Rightarrow$ & $G$ is a pivot-minor of \\  
& & a graph of path-width $\leq k+1$ \\
\hline
$G$ has rank-width $\leq 1$  & $\Leftrightarrow$ & $G$ is a vertex-minor of a tree \\ 
$G$ has linear rank-width $\leq 1$ & $\Leftrightarrow$ & $G$ is a vertex-minor of a path \\
\hline

$G$ is bipartite and has rank-width $\leq 1$ & $\Leftrightarrow$ & $G$ is a pivot-minor of a tree \\
$G$ is bipartite and has linear rank-width $\leq 1$ & $\Leftrightarrow$ & $G$ is a pivot-minor of a path \\
\hline 
\end{tabular}
\end{center}
\caption{Summary of theorems}
\end{table}

To prove the main theorem, we construct a graph having $G$ as a pivot-minor, called a rank-expansion. Then we prove that a rank-expansion has small tree-width.

The paper is organized as follows. We present the definition of rank-width and related operations in the next section. In Section 3, we define a \emph{rank-expansion} of a graph and prove the main theorem. In Section 4, using a rank-expansion, we present new characterizations of graphs with rank-width at most~$1$ and graphs with linear rank-width at most~$1$.

\section{Preliminaries} \label{sec:prelim}

In this paper, all graphs are simple and undirected. Let $G=(V,E)$ be
a graph. For a vertex $v$ of $G$, let $N(v)$ be the set of vertices adjacent to
$v$ and 
let $\delta(v)$ be the set of edges incident with $v$.
The \emph{degree} of a vertex $v$, denoted by $\deg(v)$, is defined as $\deg(v):=\abs{\delta(v)}$. For $S\subseteq V$, $G[S]$ denotes the subgraph of $G$ induced on $S$. For two sets $A$ and $B$, $A\Delta B=(A\cup B)\setminus (A\cap B)$.

A \emph{vertex partition} of a graph $G$ is a pair $(A,B)$ of subsets of $V(G)$ such that $A\cup B=V(G)$ and $A\cap B=\emptyset$. A vertex $v\in V$ is a \emph{leaf} if $\deg(v)=1$; Otherwise we call it an \emph{inner vertex}. An edge $e\in E$ is an \emph{inner edge} if $e$ does not have a leaf as an end. Let $V_I(G)$ and $E_I(G)$ be the set of inner vertices of $G$ and inner edges of $G$, respectively.

For an $X\times Y$ matrix $M$ and subsets $A\subseteq X$ and
$B\subseteq Y$, $M[A,B]$ denotes the $A\times B$ submatrix
$(m_{i,j})_{i\in A,j\in B}$ of $M$. For $a\in A$ and $b\in B$, we
denote $M_{a,b}=M[\{a\},\{b\}]$. 
If $A=B$, then $M[A]=M[A,A]$ is called a \emph{principal submatrix} of $M$. The adjacency matrix of a graph $G$, which is a $(0,1)$-matrix over the binary field, will be denoted by $A(G)$.

\subsection*{Pivoting matrices.} Let $M=
\bordermatrix{
& X & V\setminus X\cr
X & A & B\cr
V\setminus X & C & D
}
$ be a $V\times V$ matrix over a field $F$. If $A=M[X]$ is nonsingular, then we define
\[ 
M\ast X=
\bordermatrix{
 & X & V\setminus X\cr
X & A^{-1} & A^{-1}B \cr
V\setminus X & -CA^{-1}  & D-CA^{-1}B 
}
\]

This operation is called a \emph{pivot}, sometimes called a
\emph{principal pivot transformation}~\cite{Tsatsomeros2000}.
Tucker showed the following theorem.
\begin{THM}[Tucker \cite{Tucker60}]\label{thm:21}
Let $M$ be a $V\times V$ matrix over a field.
If $M[X]$ is a nonsingular principal submatrix of $M$, 
then for every subset $Y$ of $V$, 
$(M\ast X)[Y]$ is nonsingular if and only if $M[X\Delta Y]$ is nonsingular.
\end{THM}
\begin{proof}
  See Bouchet's proof in Geelen~\cite[Theorem 2.7]{Geelen1995}.
\end{proof}
The following thereom is well known, see Geelen~\cite[Theorem 2.8]{Geelen1995}.
For our purpose, we will only work on skew-symmetric matrices on the binary field
and in this case, it follows easily from Theorem~\ref{thm:21}.
\begin{THM}\label{thm:22}
  Let $M$ be a square matrix. If $M[X]$ and $M*X[Y]$ are nonsingular,
  then $(M*X)*Y=M*(X\Delta Y)$.
\end{THM}

\subsection*{Vertex-minors and pivot-minors.} The graph obtained from
$G=(V,E)$ by applying \emph{local complementation} at a vertex $v$
is \[G*v=(V,E\Delta \{ xy:xv,yv\in E ,x\neq y\} ).\] 
The graph obtained from $G$ by \emph{pivoting} an edge $uv$ is defined by $G \pivot uv=G*u*v*u$.

To see how we obtain the resulting graph by pivoting an edge $uv$, let
$V_1=N(u)\cap N(v)$, $V_2=N(u)\setminus( N(v)\cup \{v\})$ and
$V_3=N(v)\setminus (N(u)\cup \{u\})$. One can easily verify that
$G\pivot uv$ is identical to the graph obtained from $G$ by
complementing adjacency between vertices in distinct sets $V_i$ and $V_j$ and swapping the vertices $u$ and $v$~\cite{Oum05}. See Figure~\ref{fig:pivot} for example.

A graph $H$ is a \emph{vertex-minor} of $G$ if $H$ can be obtained from $G$ by applying a sequence of vertex deletions and local complementations. A graph $H$ is a \emph{pivot-minor} of $G$ if $H$ can be obtained from $G$ by applying a sequence of vertex deletions and pivoting edges. From the definition, every pivot-minor of a graph is a vertex-minor of the graph. Note that every pivot-minor of a bipartite graph is bipartite.

\begin{figure}[t]
\setlength{\unitlength}{.027in}
\begin{picture}(60,50)

\put(20,40){\line(1,0){20}}

\put(10,20){\circle*{3}}
\put(20,40){\circle*{3}}
\put(30,10){\circle*{3}}
\put(40,40){\circle*{3}}
\put(50,20){\circle*{3}}

\put(20,45){\makebox(0,0){$u$}}
\put(40,45){\makebox(0,0){$v$}}
\put(10,15){\makebox(0,0){$a$}}
\put(30,5){\makebox(0,0){$b$}}
\put(50,15){\makebox(0,0){$c$}}

\put(20,40){\line(-1,-2){10}}
\put(20,40){\line(1,-3){10}}

\put(40,40){\line(1,-2){10}}
\put(40,40){\line(-1,-3){10}}

\put(10,20){\line(2,-1){20}}
\put(7,2){\makebox(0,0){$G$}}
\end{picture}
\begin{picture}(60,50)
\put(51,2){\makebox(0,0){$G\pivot uv$}}

\put(40,40){\line(1,-2){10}}

\put(10,20){\circle*{3}}
\put(20,40){\circle*{3}}
\put(30,10){\circle*{3}}
\put(40,40){\circle*{3}}
\put(50,20){\circle*{3}}

\put(40,45){\makebox(0,0){$u$}}
\put(20,45){\makebox(0,0){$v$}}
\put(10,15){\makebox(0,0){$a$}}
\put(30,5){\makebox(0,0){$b$}}
\put(50,15){\makebox(0,0){$c$}}

\put(20,40){\line(-1,-2){10}}
\put(20,40){\line(1,-3){10}}
\put(20,40){\line(1,0){20}}

\put(40,40){\line(-1,-3){10}}

\put(10,20){\line(1,0){40}}
\put(50,20){\line(-2,-1){20}}

\end{picture}
\begin{picture}(60,50)
\put(51,2){\makebox(0,0){$G\pivot uv\pivot uc$}}

\put(40,40){\line(1,-2){10}}

\put(20,40){\line(1,0){20}}

\put(10,20){\circle*{3}}
\put(20,40){\circle*{3}}
\put(30,10){\circle*{3}}
\put(40,40){\circle*{3}}
\put(50,20){\circle*{3}}

\put(40,45){\makebox(0,0){$c$}}
\put(20,45){\makebox(0,0){$v$}}
\put(10,15){\makebox(0,0){$a$}}
\put(30,5){\makebox(0,0){$b$}}
\put(50,15){\makebox(0,0){$u$}}

\put(40,40){\line(-1,-3){10}}

\put(10,20){\line(1,0){40}}
\put(10,20){\line(2,-1){20}}
\put(50,20){\line(-2,-1){20}}

\end{picture}
\caption{Pivoting an edge $uv$. Note that $G\pivot uv\pivot uc=G\pivot vc$.}
\label{fig:pivot}\end{figure}
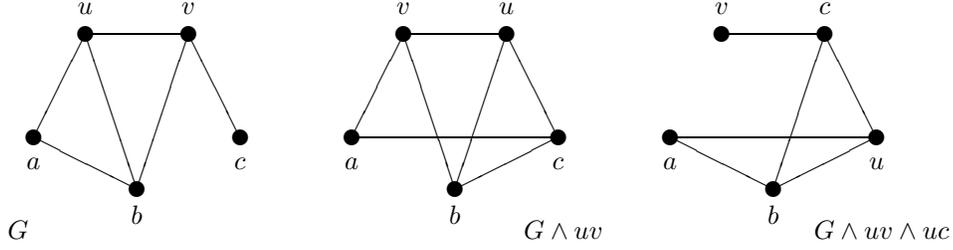

\subsection*{Pivoting in a graph is a special case of a matrix pivot.}
For a graph $G$, two vertices $u$ and $v$ are adjacent if and only if 
$\det(A(G)[\{u,v\}])\neq 0$.
This allows us to determine the graph from the list of nonsingular 
principal submatrices of $A(G)$.
If we are given the list of nonsingular principal submatrices of
$A(G)\ast X$, we can still recover the graph $G$ by Theorem~\ref{thm:21}.

In fact, if $uv\in E$, then $A(G\pivot uv)=A(G)\ast \{u, v\}$. This is
useful, because by Theorem~\ref{thm:22}, the adjacency matrix of
$H=G\pivot a_1b_1 \pivot \ldots  \pivot a_nb_n$ can be obtained by a
single pivot operation $A(G)\ast X$  where $X=\{a_1,b_1\}\Delta  \ldots  \Delta \{a_n,b_n\}$. Then $u$, $v$ are adjacent in $H$ if and only if $A(G)[X\Delta \{u,v\}]$ is nonsingular.

If $A(G)[X]$ is nonsingular, then we denote $G\pivot X$ as the graph
having the adjacency matrix $A(G)\ast X$. For $X\subseteq V(G)$, if $A(G)[X]$ is nonsingular, then we can obtain the graph $G\pivot X$ from $G$ by applying a sequence of pivoting edges, by Theorem~\ref{thm:21}. Thus, we deduce that $H$ is a pivot-minor of $G$ if and only if $H=G\pivot X\setminus Y$ where $X,Y\subseteq V(G)$ and $A(G)[X]$ is nonsingular.

\subsection*{Rank-width and linear rank-width.} The \emph{cut-rank} function $\cutrk_{G} :  2^{V} \rightarrow \mathbb{Z}$ of a graph $G=(V,E)$ is defined by 
\[
\cutrk_{G}(X)=\rank(A(G)[X,V\setminus X]).
\]

A tree is \emph{subcubic} if it has at least two vertices and every inner vertex has degree~$3$. A \emph{rank-decomposition} of a graph $G$ is a pair $(T,L)$, where $T$ is a subcubic tree and $L$ is a bijection from the vertices of $G$ to the leaves of $T$. For an edge $e$ in $T$, $T\setminus e$ induces a partition $(X_{e} ,Y_{e} )$ of the leaves of $T$. The \emph{width} of an edge $e$ is defined as $\cutrk_{G} (L^{-1}(X_{e} ))$. The \emph{width} of a rank-decomposition $(T,L)$ is the maximum width over all edges of $T$. The \emph{rank-width} of $G$, denoted by $\rw(G)$, is the minimum width of all rank-decompositions of $G$. If $\abs{V}\leq 1$, then $G$ admits no rank-decomposition and $\rw(G)=0$.

A tree is a \emph{caterpillar} if it contains a path $P$ such that every vertex of a tree has distance at most $1$ to some vertex of $P$. A \emph{linear rank-decomposition} of a graph G is a rank-decomposition $(T,L)$ of G, where $T$ is a caterpillar. The \emph{linear rank-width} of G is defined as the minimum width of all linear rank-decompositions of $G$. If $\abs{V}\leq 1$, then $G$ admits no linear rank-decomposition and $\lrw(G)=0$. Note that if a graph $H$ is a vertex-minor or a pivot-minor of a graph $G$, then $\rw(H)\leq \rw(G)$ and $\lrw(H)\leq \lrw(G)$~\cite{Oum05}. Trivially, $\rw(G)\leq \lrw(G)$.

\subsection*{Tree-width and path-width.} 
A \emph{tree-decomposition} of a graph $G=(V,E)$ is 
a pair $(T,B)$ 
of a tree $T$
and a family $B=\{B_t\}_{t\in V(T)}$ of vertex sets $B_t\subseteq V(G)$,
called \emph{bags},
satisfying the following three conditions:
\begin{enumerate}
\item[(T1)] $V(G)=\bigcup_{v\in V(T)}B_t$.
\item[(T2)] For every edge $uv$ of $G$, there exists a vertex $t$ of $T$ such that $u$, $v\in B_t$.
\item[(T3)] For $t_1$, $t_2$ and $t_3\in V(T)$, $B_{t_1}\cap B_{t_3}\subseteq B_{t_2}$ whenever $t_2$ is on the path from $t_1$ to $t_3$.
\end{enumerate}

The \emph{width} of a tree-decomposition $(T,B)$ is $\max\{ \abs{B_{t}}-1:t\in V(T)\}$. The \emph{tree-width} of $G$, denoted by $\tw(G)$, is the minimum width of all tree-decompositions of $G$. A \emph{path-decomposition} of a graph $G$ is a tree-decomposition $(T,B)$ where $T$ is a path. The \emph{path-width} of $G$, denoted by $\pw(G)$, is the minimum width of all path-decompositions of $G$.

\section{Rank-expansions and pivot-minors of graphs with small tree-width} \label{sec:rankexpansion}

In this section, we aim to construct, for a graph $G$ of rank-width $k$, a
bigger graph having tree-width at most $2k$ such that it has 
a pivot-minor isomorphic to $G$.

\begin{THM}\label{thm:main1}
Let $k$ be a non-negative integer. Let $G$ be a graph of rank-width at
most $k$ such that $\abs{V(G)}\geq 3$. Then there exists a graph $H$ having a pivot-minor isomorphic to $G$ such that tree-width of $H$ is at most $2k$ and $\abs{V(H)}\leq (2k+1)\abs{V(G)}-6k$.
\end{THM} 
 
For a graph of small linear rank-width, we can find a bigger graph
having small path-width instead of tree-width and 
reduce the upper bound on the path-width of a bigger graph as follows.

\begin{THM}\label{thm:main2}
Let $k$ be a non-negative integer. Let $G$ be a graph of linear rank-width at most $k$ and $\abs{V(G)}\geq 3$. Then there exists a graph $H$ having a pivot-minor isomorphic to $G$ such that path-width of $H$ is at most $k+1$ and $\abs{V(H)}\leq (2k+1)\abs{V(G)}-6k$. 
\end{THM}

To prove these two theorems, we need the following simple lemma on
linear algebra.

\begin{LEM}\label{lem:basis}
Let $G$ be a graph and $(A_1,B_1)$, $(A_2,B_2)$ be two vertex
partitions of $G$ such that $A_2\subseteq A_1$. Let $S\subseteq A_1$ be a set corresponding to a basis of row vectors in $A(G)[A_1,B_1]$. Then there exists a subset of $A_2$ representing a basis of row vectors in $A(G)[A_2,B_2]$ containing $S\cap A_2$. 
\end{LEM}  

\begin{proof}
Because $A_2\subseteq A_1$, row vectors in $A(G)[S\cap A_2,B_2]$ are
linearly independent. Therefore we can extend $S\cap A_2$ to a basis of rows in $A(G)[A_2,B_2]$. 
\end{proof}

\begin{figure}[t]
\setlength{\unitlength}{.03in}
\begin{picture}(60,60)

\put(10,30){\circle*{2.7}}
\put(20,50){\circle*{2.7}}
\put(20,10){\circle*{2.7}}
\put(30,25){\circle*{2.7}}
\put(40,50){\circle*{2.7}}
\put(40,10){\circle*{2.7}}
\put(50,30){\circle*{2.7}}

\put(20,55){\makebox(0,0){$a_1$}}
\put(5,30){\makebox(0,0){$a_2$}}
\put(20,5){\makebox(0,0){$a_3$}}
\put(30,20){\makebox(0,0){$a_4$}}
\put(40,55){\makebox(0,0){$a_5$}}
\put(55,30){\makebox(0,0){$a_6$}}
\put(40,5){\makebox(0,0){$a_7$}}

\put(20,50){\line(-1,-2){10}}
\put(20,50){\line(1,0){20}}
\put(20,50){\line(2,-5){10}}
\put(10,30){\line(4,-1){20}}
\put(10,30){\line(1,-2){10}}
\put(20,10){\line(1,0){20}}
\put(30,25){\line(4,1){20}}
\put(20,10){\line(2,3){10}}
\put(40,50){\line(1,-2){10}}
\put(40,10){\line(1,2){10}}
\put(20,50){\line(3,-2){30}}
\put(20,10){\line(3,2){30}}
\end{picture}
\begin{picture}(100,60)

\put(5,10){\circle*{2.7}}
\put(5,40){\circle*{2.7}}
\put(15,25){\circle*{2.7}}
\put(30,37){\circle*{2.7}}

\put(45,40){\circle*{2.7}}
\put(37,52){\circle*{2.7}}
\put(53,52){\circle*{2.7}}
\put(60,25){\circle*{2.7}}
\put(75,40){\circle*{2.7}}
\put(75,10){\circle*{2.7}}

\put(5,45){\makebox(0,0){$L(a_1)$}}
\put(5,5){\makebox(0,0){$x=L(a_2)$}}
\put(30,42){\makebox(0,0){$L(a_3)$}}
\put(36,57){\makebox(0,0){$L(a_4)$}}
\put(54,57){\makebox(0,0){$L(a_5)$}}
\put(75,45){\makebox(0,0){$L(a_6)$}}
\put(75,5){\makebox(0,0){$L(a_7)$}}

\put(5,10){\line(2,3){10}}
\put(15,25){\line(1,0){15}}
\put(30,25){\line(1,0){15}}
\put(45,25){\line(1,0){15}}
\put(60,25){\line(1,-1){15}}
\put(5,40){\line(2,-3){10}}
\put(30,25){\line(0,1){12}}
\put(45,25){\line(0,1){15}}
\put(45,40){\line(-2,3){8}}
\put(45,40){\line(2,3){8}}
\put(60,25){\line(1,1){15}}

\put(30,25){\vector(1,0){10}}
\put(30,25){\line(1,0){15}}

\put(5,10){\vector(2,3){5}}
\put(15,25){\vector(-2,3){5}}
\put(15,25){\vector(1,0){8}}
\put(30,25){\vector(0,1){7}}
\put(45,25){\vector(1,0){8}}
\put(45,25){\vector(0,1){8}}
\put(45,40){\vector(-2,3){5}}
\put(45,40){\vector(2,3){5}}
\put(60,25){\vector(1,1){7}}
\put(60,25){\vector(1,-1){7}}

\put(30,25){\circle*{2.7}}
\put(45,25){\circle*{2.7}}

\put(30,20){\makebox(0,0){$w$}}
\put(38,28){\makebox(0,0){$e$}}
\put(45,20){\makebox(0,0){$v$}}
\put(48,33){\makebox(0,0){$f_1$}}
\put(53,28){\makebox(0,0){$f_2$}}
\put(23,28){\makebox(0,0){$d$}}

\end{picture}
\caption{A graph $G$ and a rank-decomposition $(T,L)$ of $G$ with a fixed leaf $x\in V(T)$. Note that the edge $e\in E(T)$ has width $3$ and $e$ is directed from $w$ to $v$.}
\label{fig:graph}
\end{figure}
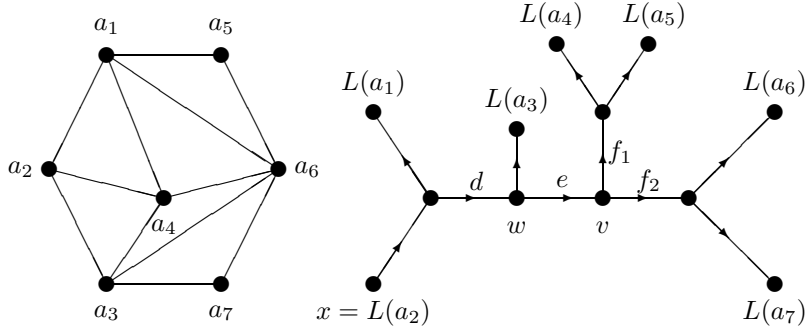

\subsection{Construction of a rank-expansion.}
To prove Theorems \ref{thm:main1} and \ref{thm:main2}, we construct a
\emph{rank-expansion} of a graph as follows. Let $G$ be a connected
graph and $(T,L)$ be a rank-decomposition of $G$ having width at most
$k$. 
We fix a leaf $x\in
V(T)$. For $e\in E(T)$, let $T_e$ be the component of $T\setminus e$
which does not contain $x$, and let $A_e=L^{-1}(V(T_e))$,
$B_e=V(G)\setminus A_e$ and $M_e=A(G)[A_e,B_e]$. For each $a\in A_e$,
let $R^e_a=M_e[\{a\},B_e]$ be the row vector of $M_e$ corresponding to $a$.

First, we orient each edge of $T$ away from $x$. By
Lemma~\ref{lem:basis}, we can choose a vertex
set $U_e\subseteq A_e$ for each edge $e$ of $T$ satisfying the
following two conditions:
\begin{enumerate}
\item $\{R^e_w\}_{w\in U_e}$ forms a basis of row vectors in~$M_e$
  for each edge $e$ of $T$.
\item $(U_e\cap A_f)\subseteq U_f$ if the tail
of an edge $f$ is the head of $e$.
\end{enumerate}
Since $(T,L)$ has width at most $k$, we have $\abs{U_e}\le k$ for each
edge $e$ of $T$.
 Since $R^e_a$ can be uniquely
expressed as a linear combination of vectors in $\{R^e_w\}_{w\in U_e}$
for each $a\in A_e$, there exists a unique $A_e\times U_e$ matrix
$P_e$ such that $P_e(A(G)[U_e,B_e])=A(G)[A_e,B_e]$. 

For example, in Figure~\ref{fig:graph}, 
\[A(G)[A_e,B_e]= 
\bordermatrix{
&a_1&a_2&a_3\cr
a_4&1&1&1\cr
a_5&1&0&0\cr
a_6&1&0&1 \cr
a_7&0&0&1 }
\]
and $\{R^e_{a_4},R^e_{a_5},R^e_{a_7} \}$  forms a basis of row vectors of $A(G)[A_e,B_e]$. So, if we let $U_e=\{a_4,a_5,a_7 \}$, then 
\[P_e= 
\bordermatrix{
&a_4&a_5&a_7\cr 
a_4&1&0&0\cr
a_5&0&1&0\cr
a_6&0&1&1 \cr
a_7&0&0&1 \cr
}
\]
and we easily verify that $P_eA(G)[U_e,B_e]=A(G)[A_e,B_e]$. 

If the tail of an edge $f$ is the head of an edge $e$, then let $C_f=P_e[U_f,U_e]$. We will use the property that if $e_{n+1}e_n \ldots e_1$ is a directed path in $T$, then \[C_{e_1}C_{e_2} \ldots C_{e_n}=P_{e_{n+1}}[U_{e_1},U_{e_{n+1}}].\]
A \emph{rank-expansion} $\boldsymbol{R}(G,T,L,x,\{U_f\}_{f\in E(T)})$
of a graph $G$
is a graph $H$ such that 
\begin{align*}
V(H)&=\bigcup_{v\in V_I(T)} S_v
\quad\text{ where }S_v=\bigcup_{e\in \delta(v)}(U_e\times \{e\} \times
\{v\}) \text{ for each }v\in V_I(T), \\
E(H)&=\{ \{(a,e,v),(a,e,w)\}:e=vw\in E_I(T), a\in U_e\} \\
    &\, \cup\{ \{(a,e,v),(b,f,v)\}: v\in V_I(T), e, f\in E(T),  v\text{ is the head of } e\text{ and the tail of } f,\\
    &\qquad\qquad\qquad\qquad\qquad\  a\in U_f, b\in U_e \text{ and } (C_f)_{a,b}\neq 0\} \\
    &\, \cup\{ \{(a,f_1,v),(b,f_2,v)\}: v\text{ is the tail of both } f_1 \text{ and } f_2\in E(T), \\
    &\,\,\,\,\qquad\qquad\qquad\qquad\qquad\ a\in U_{f_1}, b\in
    U_{f_2}\text{ and }ab\in E(G)\}.
\end{align*}
(The sets $V_I(T), E_I(T)$ are defined in the beginning of Section~\ref{sec:prelim}.)

For $e=vw\in E_I(T)$, let $\overline{e}=\{(a,e,v): a\in U_e\}\cup \{(a,e,w): a\in U_e\}\subseteq V(H)$ and for $W\subseteq E_I(T)$, let $\overline{W}=\bigcup_{e\in W} \overline{e}\subseteq V(H)$. 
If $e\in E_I(T)$ is directed from $w$ to $v$, let $L_e=S_v\cap
\overline{e}$ and $R_e=S_w\cap \overline{e}$. 
For a vertex $a$ in $V(G)$,
$T$ has a unique edge $e$ incident with $L(a)$ and some vertex $v$ of
$T$
and we write  $\overline{a}$ to denote  the unique
vertex in $U_e\times \{e\}\times \{v\}$ and let
$\overline{e}:=\overline{a}$.  Notice that since $G$ is connected,
$U_e$ is nonempty.

We discuss the number of vertices in the rank-expansion $H$. We easily observe that $\abs{E_I(T)}=\abs{V(G)}-3$. So if $\rw(G)\leq k$, then $\abs{\overline{e}}\leq 2k$ for each $e\in E_I(T)$, and we deduce that $\abs{V(H)}\leq 2k\abs{E_I(T)}+\abs{V(G)}=2k(\abs{V(G)}-3)+\abs{V(G)}=(2k+1)\abs{V(G)}-6k$.

\begin{figure}[t]
 \tikzstyle{v}=[circle, draw, solid, fill=black, inner sep=0pt, minimum width=3pt]
\begin{tikzpicture}[scale=0.056]
\node at (21,75) {$U_d=\{a_4, a_5\}$};
\node at (21,67) {$U_e=\{a_4, a_5, a_7\}$};
\node at (19,59) {$U_{f_1}=\{a_4, a_5\}$};
\node at (19,51) {$U_{f_2}=\{a_6, a_7\}$};

\draw (125,26) -- (145,45);
\draw (125,18) -- (153,45);
\draw (125,10) -- (173,10);

\draw (48,18) -- (20,3) -- (20,25) -- (48,10);
\draw (20,25) -- (48,18);

\node [left] at (20,25) {$\overline{a_1}$};
\node [left] at (20,3) {$\overline{a_2}(a_4)$};
\node [v] at (20,25) {};
\node [v] at (20,3) {};


\draw [blue][very thick](48,10) -- (63,10);
\draw [blue][very thick](48,18) -- (63,18);

\node [v] at (48,10) {}; 
\node [v] at (48,18) {};
\node [v] at (63,10) {}; \node [v] at (63,18) {};

\node at (55,21) {$a_4$};
\node at (55,13) {$a_5$};
\node [above] at (87,45) {$\overline{a_3}$};
\node [v] at (87,45) {};
\node at (55,2) {$\overline{d}$};

\draw (63,18) -- (87,45) -- (63,10);
\draw (110,26) -- (87,45) -- (110,10);
\draw (63,18) -- (110,26);
\draw (63,10) -- (110,18);

\draw [blue][very thick] (110,10) -- (125,10);
\draw [blue][very thick] (110,18) -- (125,18);
\draw [blue][very thick] (110,26) -- (125,26);
\draw [blue][very thick] (145,45) -- (145,60);
\draw [blue][very thick] (153,45) -- (153,60);
\draw [blue][very thick] (173,10) -- (188,10);
\draw [blue][very thick] (173,18) -- (188,18);

\node [v] at (110,10) {}; \node [v] at (110,18) {}; \node [v] at (110,26) {};
\node [v] at (125,10) {}; \node [v] at (125,18) {}; \node [v] at (125,26) {};
\node [v] at (145,45) {}; \node [v] at (153,45) {};
\node [v] at (145,60) {}; \node [v] at (153,60) {};
\node [v] at (173,10) {}; \node [v] at (173,18) {};
\node [v] at (188,10) {}; \node [v] at (188,18) {};

\node at (117,29) {$a_4$};
\node at (117,21) {$a_5$};
\node at (117,13) {$a_7$};
\node at (180,21) {$a_6$};
\node at (180,13) {$a_7$};
\node at (141,53) {$a_4$};
\node at (150,53) {$a_5$};

\node at (117,2) {$\overline{e}$};
\node at (180,2) {$\overline{f_2}$};
\node at (162,53) {$\overline{f_1}$};
\node at (110,35) {$R_e$};
\node at (125,35) {$L_e$};

\draw (107.5,7.5) rectangle (112.5,30.5);
\draw (122.5,7.5) rectangle (127.5,30.5);

\draw (145,45)-- (173,18);
\draw (153,45)-- (173,18);
\draw (125,18)-- (173,18);
\draw (173,18)-- (125,10);

\node [left] at (137,80) {$\overline{a_4}$};
\node [right] at (159,80) {$\overline{a_5}$};
\node [v] at (137,80) {};
\node [v] at (159,80) {};

\draw (145,60)-- (137,80);
\draw (153,60)-- (159,80);

\node [right] at (216,3) {$\overline{a_7}$};
\node [right] at (216,25) {$\overline{a_6}$};
\node [v] at (216,3) {};
\node [v] at (216,25) {};

\draw (188,10)-- (216,3);
\draw (188,18)-- (216,25)-- (216,3);
\end{tikzpicture}
\caption{A rank-expansion of the graph $G$ in Figure~\ref{fig:graph}. }
\label{fig:exp}
\end{figure}
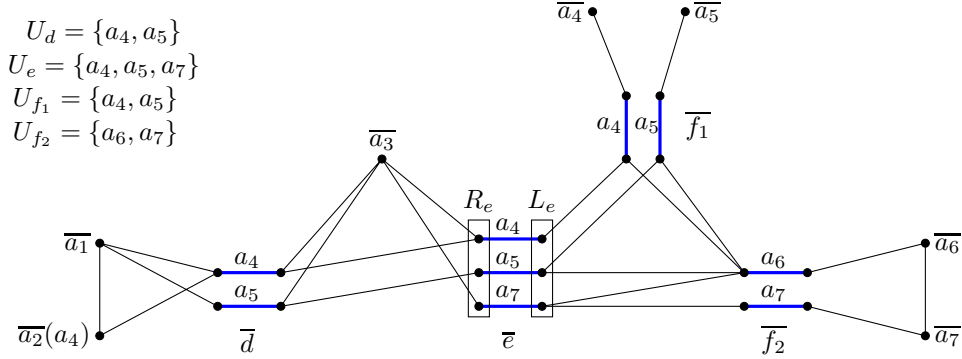

\subsection{A graph is a pivot-minor of its rank-expansion.}
First, we prove that every rank-expansion of a graph $G$ has a
pivot-minor isomorphic to $G$. To obtain $G$ as a pivot-minor of a
rank-expansion $H$,
we will prove that $H\pivot \overline{E_I(T)}$ has an induced subgraph
isomorphic to $G$. We first need to verify that
$A(H)[\overline{E_I(T)}]$ is nonsingular in order to apply the matrix pivot.

\begin{LEM}\label{lem:deg1}
Let $G$ be a graph and $uv\in E(G)$. If $\deg(u)=1$, then $G\pivot uv\setminus \{u,v\}=G\setminus \{u,v\}$.
\end{LEM}
\begin{proof}
It is clear from the definition.
\end{proof}

\begin{LEM}\label{lem:ignore1}
The matrix $A(H)[\overline{E_I(T)}]$ is nonsingular.
\end{LEM}
\begin{proof}
We claim that for all $W\subseteq E_I(T)$, $A(H)[\overline{W}]$ is nonsingular. We proceed by induction on $\abs{W}$. If $W$ is empty, then it is trivial. If $\abs{W}\geq 1$, then $W$ induces a forest in $T$, and therefore there must be an edge $f\in W$ which has a leaf in $T[W]$. By induction hypothesis, $A(H)[\overline{W\setminus \{f\}}]$ is nonsingular. Since every edge in $H[\overline{f}]$ is incident with a leaf in $H[\overline{W}]$, by Lemma~\ref{lem:deg1}, pivoting all edges in $\overline{f}$ does not change the graph $H[\overline{W\setminus \{f\}}]$. So, $A(H[\overline{W}]\pivot \overline{f})[\overline{W\setminus \{f\}}]=A(H)[\overline{W\setminus \{f\}}]$ and therefore, by Theorem~\ref{thm:21}, $A(H)[\overline{f}\Delta \overline{W\setminus \{f\}}]=A(H)[\overline{W}]$ is nonsingular.
\end{proof}

By Lemma~\ref{lem:ignore1}, we can pivot $H$ by
$\overline{E_I(T)}$. Now in order to determine the adjacency in the
graph $H\pivot \overline{E_I(T)}$, we need to determine whether the
matrix $A(H)[\overline{E_I(T)}\cup \{\overline{a}, \overline{b}\}]$ is
nonsingular where $a$, $b\in V(G)$. 
In the following lemma, we will show that to determine the adjacency in the graph $H\pivot \overline{E_I(T)}$, it is enough to pivot a small set of vertices.

\begin{LEM}\label{lem:ignore2}
Let $a$, $b\in V(G)$ and let $P$ be a path from $L(a)$ to $L(b)$ in $T$. Then $A(H)[\overline{E_I(T)}\cup \{\overline{a}, \overline{b}\}]$ is nonsingular if and only if $A(H)[\overline{E(P)}]$ is nonsingular.
\end{LEM}

\begin{proof}
We claim that for $E(P)\cap E_I(T)\subseteq W\subseteq E_I(T)$, $A(H)[\overline{W}\cup \{\overline{a}, \overline{b}\}]$ is nonsingular if and only if $A(H)[\overline{E(P)}]$ is nonsingular.

We use induction on $\abs{W}$. If $W=E(P)\cap E_I(T)$, then it is trivial, because $\overline{W}\cup \{\overline{a},\overline{b}\}=\overline{E(P)}$. So we may assume that $\abs{W}>\abs{E(P)\cap E_I(T)}$. Since $P$ is a maximal path in $T$, the subgraph of $T$ having the edge set $W\cup E(P)$ must have at least $3$ leaves. Thus there is an edge $f$ in $W\setminus E(P)$ incident with a leaf in $T[W\cup E(P)]$ other than $L(a)$ and $L(b)$. Since every edge in $\overline{f}$ is incident with a leaf in $H[\overline{W}]$, by Lemma~\ref{lem:deg1}, $A(H[\overline{W}\cup \{\overline{a}, \overline{b}\}]\pivot \overline{f})[\overline{W\setminus \{f\}}\cup \{\overline{a}, \overline{b}\}]=A(H)[\overline{W\setminus \{f\}}\cup \{\overline{a}, \overline{b}\}]$. By induction hypothesis and Theorem~\ref{thm:21}, we deduce that 
\begin{align*}
A(H)[\overline{E(P)}]\text{ is nonsingular}
&\Leftrightarrow A(H)[\overline{W\setminus \{f\}}\cup \{\overline{a}, \overline{b}\}]\text{ is nonsingular}\\
&\Leftrightarrow A(H[\overline{W}\cup \{\overline{a}, \overline{b}\}]\pivot \overline{f})[\overline{W\setminus \{f\}}\cup \{\overline{a}, \overline{b}\}]\text{ is nonsingular}\\
&\Leftrightarrow A(H)[\overline{W}\cup \{\overline{a}, \overline{b}\}]\text{ is nonsingular.}\qedhere
\end{align*}\end{proof}

From now on, we focus on how to determine the adjacency in $H\pivot \overline{E_I(T)}$ by computing $\det\left(A(H)[\overline{E(P)}]\right)$.

\begin{LEM}\label{lem:path}
Let $P=(e_{n+1}, e_n, \ldots, e_1)$ be the directed path from $w$ to $v$ in $T$. Then $C_{e_1}C_{e_2} \ldots C_{e_n}A(G)[U_{e_{n+1}}, B_{e_{n+1}}]=A(G)[U_{e_1},B_{e_{n+1}}]$.
\end{LEM}

\begin{proof}
We proceed by induction on $n$. If $n=1$, then by definition, 
\[C_{e_1}A(G)[U_{e_2},B_{e_2}]=P_{e_2}[U_{e_1},U_{e_2}]A(G)[U_{e_2},B_{e_2}]=A(G)[U_{e_1},B_{e_2}].\]
We may assume that $n\geq 2$. By induction hypothesis, 
\[C_{e_2}C_{e_3} \ldots C_{e_n}A(G)[U_{e_{n+1}},B_{e_{n+1}}]=A(G)[U_{e_2},B_{e_{n+1}}].\]
Since $C_{e_1}A(G)[U_{e_2},B_{e_2}]=A(G)[U_{e_1},B_{e_2}]$ and $B_{e_{n+1}}\subseteq B_{e_2}$, 
\[C_{e_1}A(G)[U_{e_2},B_{e_{n+1}}]=A(G)[U_{e_1},B_{e_{n+1}}].\] Therefore, we conclude that 
\begin{align*}
C_{e_1}C_{e_2} \ldots C_{e_n}A(G)[U_{e_{n+1}}, B_{e_{n+1}}]
&=C_{e_1}A(G)[U_{e_2},B_{e_{n+1}}] \\
&=A(G)[U_{e_1},B_{e_{n+1}}].\qedhere
\end{align*}
\end{proof}

\begin{LEM}\label{lem:det}
\[
\det \, \begin{blockarray}{ccccccccc}
\begin{block}{(c|cccccccc)}
0&C_1&0&0&\cdots &0&0\\ \cline{1-9}
0&I&C_2&0&\cdots &0&0\\ 
0&0&I&C_3&      &0&0\\
0&0&0&I&      &0&0\\
\vdots & & & & \ddots  & & \vdots \\
0&0&0&0&  \cdots    &I&C_n \\  
C_{n+1}&0&0&0&  \cdots    &0&I \\  
\end{block}
\end{blockarray} =(-1)^{n}\det(C_1 C_2 \ldots C_{n+1}).
\]
$($Since we mainly focus on the binary field, $-1=+1.)$
\end{LEM}

\begin{proof}
By elementary row operation,
\begin{align*}
\lefteqn{\det \,
\begin{blockarray}{cccccccccc}
\begin{block}{(c|ccccccccc)}
0&C_1&0&0&\cdots &0&0\\ \cline{1-9}
0&I&C_2&0&\cdots &0&0\\ 
0&0&I&C_3&      &0&0\\
0&0&0&I&      &0&0\\
\vdots & & & & \ddots  & & \vdots \\
0&0&0&0&  \cdots    &I&C_n \\  
C_{n+1}&0&0&0&  \cdots    &0&I \\  
\end{block}
\end{blockarray}}\displaybreak[0]\\ 
& = \det \,  
\begin{blockarray}{cccccccccc}
\begin{block}{(c|ccccccccc)}
0&0&-C_1C_2&0&\cdots &0&0\\ \cline{1-9}
0&I&C_2&0&\cdots &0&0\\ 
0&0&I&C_3&      &0&0\\
0&0&0&I&      &0&0\\
\vdots & & & & \ddots  & & \vdots \\
0&0&0&0&  \cdots    &I&C_n \\  
C_{n+1}&0&0&0&  \cdots    &0&I \\  
\end{block}
\end{blockarray}\displaybreak[0]\\ 
&= \det \,  
\begin{blockarray}{cccccccccc}
\begin{block}{(c|ccccccccc)}
0&0&0&(-1)^2C_1C_2C_3&\cdots &0&0\\ \cline{1-9}
0&I&C_2&0&\cdots &0&0\\ 
0&0&I&C_3&      &0&0\\
0&0&0&I&      &0&0\\
\vdots & & & & \ddots  & & \vdots \\
0&0&0&0&  \cdots    &I&C_n \\  
C_{n+1}&0&0&0&  \cdots    &0&I \\  
\end{block}
\end{blockarray}\displaybreak[0]\\ 
&= \det \,
\begin{blockarray}{cccccccccc}
\begin{block}{(c|ccccccccc)}
(-1)^nC_1C_2\ldots C_{n+1}&0&0&0&\cdots &0&0\\ \cline{1-9}
0&I&C_2&0&\cdots &0&0\\ 
0&0&I&C_3&      &0&0\\
0&0&0&I&      &0&0\\
\vdots & & & & \ddots  & & \vdots \\
0&0&0&0&  \cdots    &I&C_n \\  
C_{n+1}&0&0&0&  \cdots    &0&I \\  
\end{block}
\end{blockarray}\\
& = (-1)^n\det(C_1C_2\ldots C_{n+1}). \qedhere
\end{align*} 
\end{proof}

\begin{PROP}\label{prop:pivotminor}
Let $k\geq 1$. Let $G$ be a connected graph with rank-width $k$ and $\abs{V(G)}\geq 3$. Then a rank-expansion of $G$ has a pivot-minor isomorphic to $G$.
\end{PROP}

\begin{proof}
Let $(T,L)$ be a rank-decomposition of a graph $G$ and let $x$ be a leaf in $T$. We orient each edge $f$ away from $x$. For each $f\in E(T)$, if $m$ is the width of $f$, we choose a basis $U_f=\{ u^f_1, u^f_2, \ldots, u^f_m\}\subseteq A_f$ of rows in the matrix $A(G)[A_f,B_f]$ such that $(U_e\cap A_f)\subseteq U_f$ if the head of an edge $e$ is the tail of $f$. Since $G$ is connected, $\abs{U_f}\geq 1$. Let $H$ be a rank-expansion $\boldsymbol{R}(G,T,L,x,\{U_f\}_{f\in E(T)})$ of a graph $G$. By Lemma~\ref{lem:ignore1}, $A(H)[\overline{E_I(T)}]$ is nonsingular. We will prove that for $a$, $b\in V(G)$, $\overline{a}\overline{b}\in E(H\pivot \overline{E_I(T)})$ if and only if $ab\in E(G)$.

Let $a$, $b$ be distinct vertices in $G$. We consider the path $P$ from $L(a)$ to $L(b)$ in $T$. By Theorem~\ref{thm:21} and Lemma~\ref{lem:ignore2}, 
\begin{multline*}
\left(A(H\pivot \overline{E_I(T)})\right)_{\overline{a},\overline{b}} 
=\det\left(A(H\pivot \overline{E_I(T)})[\{\overline{a},\overline{b}\}]\right) \\
=\det\left(A(H)[\overline{E_I(T)}\Delta \{ \overline{a},\overline{b}\}]\right)
=\det\left(A(H)[\overline{E(P)}]\right).
\end{multline*}
Thus, it is enough to show that $\det(A(H[\overline{E(P)}]))=\left(A(G)\right)_{a,b}$.

If $L(b)=x$, then $P=(e_{n+1}, e_n, \ldots ,e_1, e_0)$ is a directed path from $L(b)$ to  $L(a)$. The submatrix of $A(H)$ induced by $\overline{E(P)}$ is 
\begin{multline*} 
\begin{blockarray}{ccccccccccccc}
&\overline{b}&L_{e_1}&L_{e_2}&\cdots&L_{e_{n-1}}&L_{e_n}&\overline{a}&R_{e_1}&R_{e_2}&\cdots&R_{e_{n-1}}&R_{e_n}\\ 
\begin{block}{c(c|ccccc|c|ccccc)}
\overline{a}  &0&C_{e_0}&0&\cdots &0&0&0&0&0&  \cdots    &0&0\\ \cline{2-13}
R_{e_1}&0&I&C_{e_1}&\cdots &0&0&0&0&0&  \cdots    &0&0\\ 
R_{e_2}&0&0&I&      &0&0&0&0&0&  \cdots    &0&0\\
\vdots&\vdots & & & \ddots  & & \vdots &0& & & \ddots  & & \vdots \\
R_{e_{n-1}}&0&0&0&  \cdots    &I&C_{e_{n-1}}&0&0&0&  \cdots    &0&0 \\  
R_{e_n}&C_{e_{n}}&0&0&  \cdots    &0&I&0&0&0&  \cdots    &0&0 \\   \cline{2-13}
\overline{b}&0&0&0&  \cdots    &0&0&0&0&0&  \cdots    &0&C^t_{e_{n}} \\   \cline{2-13}
L_{e_1}&0&0&0&  \cdots    &0&0&C^t_{e_0}&I&0&\cdots &0&0\\ 
L_{e_2}&0&0&0&  \cdots    &0&0&0&C^t_{e_1}&I&      &0&0\\
\vdots&\vdots & & & \ddots  & & \vdots &0& & & \ddots  & & \vdots\\
L_{e_{n-1}}&0&0&0&  \cdots    &0&0&0&0&0&  \cdots    &I&0 \\  
L_{e_n}&0&0&0&  \cdots    &0&0&0&0&0&  \cdots    &C^t_{e_{n-1}}&I \\  
\end{block}
\end{blockarray} \\
=\left( \begin{array}{c|c} C & 0  \\ \hline
                           0 & C^t  \end{array} \right).
\end{multline*} 

Note that $\det(A(H)[\overline{E(P)}])=\det(C)\det(C^t)=\det(C)^2$. By Lemma~\ref{lem:det}, 
\[\det(C)=(-1)^n\det(C_{e_0}C_{e_1}\ldots C_{e_n}).\] 
Since $\abs{U_{e_{n+1}}}=\abs{B_{e_{n+1}}}=1$ and $\rank(A(G)[U_e,B_e])=\abs{U_e}$ for all edges $e\in E(T)$, $A(G)[U_{e_{n+1}},B_{e_{n+1}}]=(1)$. By Lemma~\ref{lem:path}, 
\begin{align*}
C_{e_0}C_{e_1} \ldots C_{e_n}&=C_{e_0}C_{e_1} \ldots C_{e_n}A(G)[U_{e_{n+1}}, B_{e_{n+1}}]\\
&=A(G)[U_{e_0},B_{e_{n+1}}]\\
&=\left(A(G)\right)_{a,b}.
\end{align*}
Therefore $\det(A(H)[\overline{E(P)}])=\left(A(G)\right)_{a,b}$, as required.

Now we assume that $L(a)\neq x$ and $L(b)\neq x$. Then there exists a vertex $y$ in $V(P)$ such that it has a shortest distance to $x$. Let $P_1=(e_n, e_{n-1}, \ldots, e_0)$ be the edges of $P$ from $y$ to $L(a)$ and $P_2=(f_{m}, f_{m-1}, \ldots, f_0)$ be the edges of $P$ from $y$ to $L(b)$.

Let $M=A(H)[R_{e_n},R_{f_m}]$. By the construction of a rank-expansion, $M=A(G)[U_{e_n},U_{f_m}]$. The submatrix of $A(H)$ induced by $\overline{E(P)}$ is 
\[
\begin{blockarray}{ccccccc}
&\{\overline{b}\}\cup \bigcup^n_{i=1} L_{e_i}\cup \bigcup^m_{i=1} R_{f_i}&\{\overline{a}\}\cup \bigcup^n_{i=1} R_{e_i}\cup \bigcup^m_{i=1} L_{f_i}\\ 
\begin{block}{c(c|ccc|cc)}
\{\overline{a}\}\cup \bigcup^n_{i=1} R_{e_i}\cup \bigcup^m_{i=1} L_{f_i}&C&0\\  \cline{2-3}
\{\overline{b}\}\cup \bigcup^n_{i=1} L_{e_i}\cup \bigcup^m_{i=1} R_{f_i}&0&C^t\\ 
\end{block}
\end{blockarray} 
\]       
where $C$ is 
\[
\begin{blockarray}{ccccccccccccc}
&\overline{b}&L_{e_1}&L_{e_2}&\cdots&L_{e_{n-1}}&L_{e_n}&R_{f_m}&R_{f_{m-1}}&\cdots&R_{f_2}&R_{f_1}\\ 
\begin{block}{c(c|ccccc|cccccc)}
\overline{a}  &0&C_{e_0}&0&\cdots &0&0&0&0&  \cdots    &0&0\\ \cline{2-12}
R_{e_1}&0&I&C_{e_1}&\cdots &0&0&0&0&  \cdots    &0&0\\ 
R_{e_2}&0&0&I&      &0&0&0&0&  \cdots    &0&0\\
\vdots&\vdots & & & \ddots  & & \vdots & & & \ddots  & & \vdots \\
R_{e_{n-1}}&0&0&0&  \cdots    &I&C_{e_{n-1}}&0&0&  \cdots    &0&0 \\  
R_{e_n}&0&0&0&  \cdots    &0&I&M&0&  \cdots    &0&0 \\   \cline{2-12}
L_{f_{m}}&0&0&0&  \cdots    &0&0&I&C^t_{f_{m-1}}&\cdots &0&0\\ 
L_{f_{m-1}}&0&0&0&  \cdots    &0&0&0&I&      &0&0\\
\vdots&\vdots & & & \ddots  & & \vdots & & & \ddots  & & \vdots\\
L_{f_2}&0&0&0&  \cdots    &0&0&0&0&  \cdots    &I&C^t_{f_1} \\  
L_{f_1}&C^t_{f_0}&0&0&  \cdots    &0&0&0&0&  \cdots    &0&I \\
\end{block}
\end{blockarray}\, .
\]

It is enough to show that $C_{e_0}C_{e_1}\ldots C_{e_{n-1}}MC^t_{f_{m-1}}C^t_{f_{m-2}}\ldots C^t_{f_0}=A(G)(a,b)$. Since $M=A(G)[U_{e_n},U_{f_m}]\subseteq A(G)[U_{e_n},B_{e_n}]$, by Lemma~\ref{lem:path}, we have
\begin{align*}
\lefteqn{C_{e_0}C_{e_1}\ldots C_{e_{n-1}}MC^t_{f_{m-1}}C^t_{f_{m-2}}\ldots C^t_{f_0}}\\
&=C_{e_0}C_{e_1}\ldots C_{e_{n-1}}A(G)[U_{e_n},U_{f_m}]C^t_{f_{m-1}}C^t_{f_{m-2}}\ldots C^t_{f_0}\\
&=A(G)[U_{e_0},U_{f_m}]C^t_{f_{m-1}}C^t_{f_{m-2}}\ldots C^t_{f_0}\\
&=(C_{f_0}C_{f_1}\ldots C_{f_{m-1}}A(G)[U_{f_m},U_{e_0}])^t\\
&=A(G)[U_{f_0},U_{e_0}]^t=\left(A(G)\right)_{a,b}.
\end{align*}

So, $\det(A(H)[\overline{E(P)}])=\left(A(G)\right)_{a,b}$, as claimed. Therefore, $\overline{a}\overline{b}\in E(H\pivot \overline{E_I(T)})$ if and only if $ab\in E(G)$. We conclude that a rank-expansion of $G$ has a pivot-minor isomorphic to $G$.
\end{proof}

\subsection{A rank-expansion has small tree-width.}
In the next proposition, we show that a rank-expansion has tree-width at most $2k$ when $\rw(G)\leq k$.

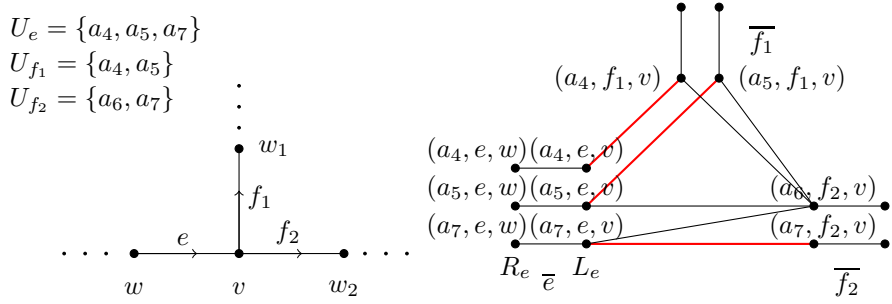
\begin{figure}[t]
\tikzstyle{v}=[circle, draw, solid, fill=black, inner sep=0pt, minimum width=3pt]
\tikzstyle{w}=[circle, draw, solid, fill=black, inner sep=0pt, minimum width=1pt]
\begin{tikzpicture}[scale=0.093]

\foreach \x in {0,3,6}
	\node [w] at (\x,15) {};
\foreach \x in {33,36,39}
	\node [w] at (25,\x) {};
\foreach \x in {43,46,49}
	\node [w] at (\x,15) {};

\draw [->] (10,15) -- (19,15);
\draw [->] (25,15) -- (34,15);
\draw [->] (25,15) -- (25,24);

\draw (10,15) -- (25,15);
\draw (25,15) -- (40,15);
\draw (25,15) -- (25,30);

\node [v] at (10,15) {};
\node [v] at (25,15) {};
\node [v] at (40,15) {};
\node [v] at (25,30) {};

\node at (17,17) {$e$};
\node at (28,23) {$f_1$};
\node at (32,18) {$f_2$};
\node at (25,10) {$v$};
\node at (10,10) {$w$};
\node at (30,30) {$w_1$};
\node at (40,10) {$w_2$};

\node at (6,47) {$U_e=\{a_4, a_5, a_7\}$};
\node at (4,42) {$U_{f_1}=\{a_4, a_5\}$};
\node at (4,37) {$U_{f_2}=\{a_6, a_7\}$};
\end{tikzpicture}
\begin{tikzpicture}[scale=0.063]

\draw [red][thick] (25,26) -- (45,45);
\draw [red][thick] (25,18) -- (53,45);
\draw [red][thick] (25,10) -- (73,10);
\draw (10,10) -- (25,10);
\draw (10,18) -- (25,18);
\draw (10,26) -- (25,26);
\draw (45,45) -- (45,60);
\draw (53,45) -- (53,60);
\draw (73,10) -- (88,10);
\draw (73,18) -- (88,18);

\node [v] at (10,10) {}; \node [v] at (10,18) {}; \node [v] at (10,26) {};
\node [v] at (25,10) {}; \node [v] at (25,18) {}; \node [v] at (25,26) {};
\node [v] at (45,45) {}; \node [v] at (53,45) {};
\node [v] at (45,60) {}; \node [v] at (53,60) {};
\node [v] at (73,10) {}; \node [v] at (73,18) {};
\node [v] at (88,10) {}; \node [v] at (88,18) {};

\node [above] at (2,9) {$(a_7,e,w)$}; 
\node [above] at (2,17) {$(a_5,e,w)$}; 
\node [above] at (2,25) {$(a_4,e,w)$};
\node [above] at (23,9) {$(a_7,e,v)$}; 
\node [above] at (23,17) {$(a_5,e,v)$}; 
\node [above] at (23,25) {$(a_4,e,v)$};
\node [left] at (43,45) {$(a_4,f_1,v)$}; 
\node [right] at (55,45) {$(a_5,f_1,v)$};
\node [above] at (75,9) {$(a_7,f_2,v)$}; 
\node [above] at (75,17) {$(a_6,f_2,v)$};

\node at (17,2) {$\overline{e}$};
\node at (80,2) {$\overline{f_2}$};
\node at (62,53) {$\overline{f_1}$};
\node at (10,5) {$R_e$};
\node at (25,5) {$L_e$};

\draw (45,45)-- (73,18);
\draw (53,45)-- (73,18);
\draw (25,18)-- (73,18);
\draw (73,18)-- (25,10);

\end{tikzpicture}
\caption{A rank-expansion of the graph $G$ in Figure~\ref{fig:graph}. By the construction of a rank-expansion, every vertex in $L_e$ has exactly one neighbor in $R_{f_1}\cup R_{f_2}\setminus \{(a_6,f_2,v)\}$ in the subgraph $H[S_v]$. }
\label{fig:expansion}\end{figure}

\begin{PROP}\label{prop:treewidth}
Let $k\geq 1$. Let $G$ be a connected graph with $\abs{V(G)}\geq 3$. If $G$ has rank-width $k$, Then $G$ has a rank-expansion of tree-width at most $2k$. Moreover, if $G$ has linear rank-width $k$, then $G$ has a rank-expansion of path-width at most $k+1$.
\end{PROP}

\begin{proof}
Let $(T,L)$ be a rank-decomposition of $G$ of width $k$. We fix a leaf $x\in V(T)$ and orient each edge $f$ away from $x$. For each $f\in E(T)$, if $m$ is the width of $f$, we choose a basis $U_f=\{ u^f_1, u^f_2, \ldots, u^f_m\}\subseteq A_f$ of rows in the matrix $A(G)[A_f,B_f]$ such that $(U_e\cap A_f)\subseteq U_f$ if the head of an edge $e$ is the tail of $f$. Since $G$ is connected, $\abs{U_f}\geq 1$. Let $H$ be a rank-expansion $\boldsymbol{R}(G,T,L,x,\{U_f\}_{f\in E(T)})$ of a graph $G$.

Let $T'$ be a tree obtained from $T[V_I(T)]$ by replacing each edge from $w$ to $v$ with a path $wz^v_1z^v_2 \ldots z^v_{\abs{U_e}}p^v_1p^v_2 \ldots p^v_{\abs{U_e}}v$. Let $y$ be the neighbor of $x$ in $T$ and let $B(y)=S_y$. For $v\in V_I(T)\setminus \{y\}$, let $e=vw$ be the edge incoming to $v$ and $f_1$, $f_2$ be edges outgoing from $v$. Let $R^v=\{(a,f,v)\in R_{f_1}\cup R_{f_2}:a\notin U_e\}$. Since $(U_e\cap A_{f_i})\subseteq U_{f_i}$ for each $i\in\{1,2\}$, each vertex in $L_e$ has exactly one neighbor in $R_{f_1}\cup R_{f_2}\setminus R^v$. Let $B(v)=R_{f_1}\cup R_{f_2}$ and $B(z^v_1)=R_e\cup \{(u^e_1,e,v)\}$, $B(p^v_1)=R^v\cup L_e\cup \{(a,f,v)\in R_{f_1}\cup R_{f_2}:a=u^e_1\}$.
And for each $2\leq i\leq \abs{U_e}$, we define 
\begin{align*}
\lefteqn{B(z^v_{i})=B(z^v_{i-1})\setminus \{(u^e_{i-1},e,w)\} \cup\{(u^e_i,e,v)\}}\\
&B(p^v_i)=B(p^v_{i-1})\setminus \{(u^e_{i-1},e,v)\} \cup\{(a,f,v)\in R_{f_1}\cup R_{f_2}:a=u^e_i\}.
\end{align*}

 Now we show that the pair $(T',\{B(v)\}_{v\in V(T')})$ is a tree-decomposition of $H$. Note that for each $v\in V_I(T)\setminus \{y\}$ with the incoming edge $e$, $\bigcup_i E(H[B(z^v_i)])=E(H[\overline{e}])$ and $\bigcup_i E(H[B(p^v_i)])=E(H[S_v])$. Therefore all vertices and all edges in $H$ are covered by $B(v)$ for some $v\in V(T')$. So the first and second axioms of a tree-decomposition are satisfied.

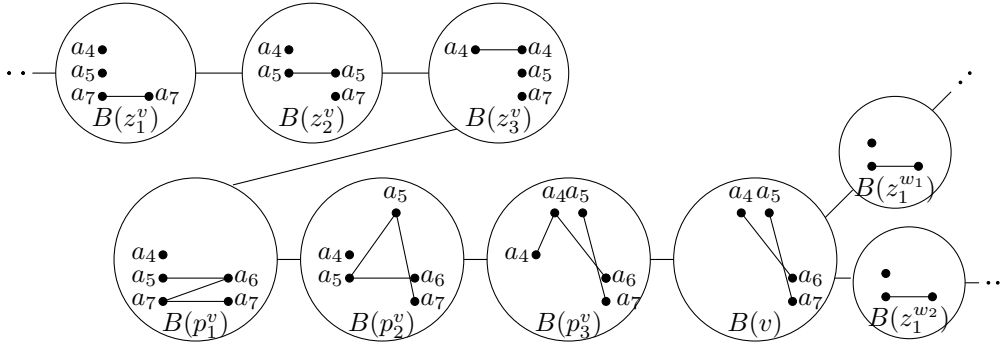
\begin{figure}[t]
\tikzstyle{v}=[circle, draw, solid, fill=black, inner sep=0pt,
minimum width=3pt]
\tikzstyle{w}=[circle, draw, solid, fill=black, inner sep=0pt,
minimum width=1pt]
\begin{tikzpicture}[scale=0.062]

\foreach \x in {20,60,100}
	\draw (\x,60) circle (15);
\foreach \x in {-5,-2}
	\node [w] at (\x,60) {};
\draw (0,60) -- (5,60);

\draw (35,60) -- (45,60);
\draw (75,60) -- (85,60);
\draw (91,48) -- (43,36);

\foreach \x in {52.5,92.5,132.5}
	\draw (\x,20) -- +(5,0);

\foreach \x in {55,60,65}
	\node [v] at (15,\x) {};
\draw (15,55) -- (25,55);
\node [v] at (25,55) {};

\node at (11,55) {$a_7$};
\node at (11,60) {$a_5$};
\node at (11,65) {$a_4$};
\node at (29,55) {$a_7$};
\node at (20,50) {$B(z^v_1)$};

\foreach \x in {52.5,92.5,132.5}
	\draw (\x,20) -- +(5,0);

\foreach \x in {60,65}
	\node [v] at (55,\x) {};
\foreach \x in {55,60}
	\node [v] at (65,\x) {};
\draw (55,60) -- (65,60);

\node at (51,60) {$a_5$};
\node at (51,65) {$a_4$};
\node at (69,55) {$a_7$};
\node at (69,60) {$a_5$};
\node at (60,50) {$B(z^v_2)$};


\foreach \x in {65}
	\node [v] at (95,\x) {};
\foreach \x in {55,60,65}
	\node [v] at (105,\x) {};
\draw (95,65) -- (105,65);

\node at (91,65) {$a_4$};
\node at (109,55) {$a_7$};
\node at (109,60) {$a_5$};
\node at (109,65) {$a_4$};
\node at (100,50) {$B(z^v_3)$};

\foreach \x in {35,75,115,155}
	\draw (\x,20) circle (17.5);

\foreach \x in {11,16,21}
	\node [v] at (28,\x) {};
\foreach \x in {11}
	\node [v] at (42,\x) {};
\draw (28,11) -- (42,11);

\node at (24,11) {$a_7$};
\node at (24,16) {$a_5$};
\node at (24,21) {$a_4$};
\node at (46,11) {$a_7$};
\node at (46,16) {$a_6$};
\node at (35,6) {$B(p^v_1)$};

\foreach \x in {42,82,123}
	\node [v] at (\x,16) {};
\foreach \x in {16,21}
	\node [v] at (68,\x) {};
	\node [v] at (82,11) {};
\draw (28,11) -- (42,16) -- (28,16);
\draw (68,16) -- (82,16);
\node [v] at (78,30) {};
\draw (68,16) -- (78,30) -- (82,11);

\node at (64,16) {$a_5$};
\node at (64,21) {$a_4$};
\node at (86,11) {$a_7$};
\node at (86,16) {$a_6$};
\node at (78,34) {$a_5$};
\node at (75,6) {$B(p^v_2)$};

\foreach \x in {21}
	\node [v] at (108,\x) {};
\foreach \x in {11,16}
	\node [v] at (123,\x) {};
\foreach \x in {112,118}
	\node [v] at (\x,30) {};
\draw (108,21) -- (112,30) -- (123,16);
\draw (118,30) -- (123,11);

\node at (104,21) {$a_4$};
\node at (128,11) {$a_7$};
\node at (127,16) {$a_6$};
\node at (117,34) {$a_5$};
\node at (112,34) {$a_4$};
\node at (115,6) {$B(p^v_3)$};

\foreach \x in {11,16}
	\node [v] at (163,\x) {};
\foreach \x in {152,158}
	\node [v] at (\x,30) {};
\draw (152,30) -- (163,16);
\draw (158,30) -- (163,11);

\node at (167,11) {$a_7$};
\node at (167,16) {$a_6$};
\node at (158,34) {$a_5$};
\node at (152,34) {$a_4$};
\node at (155,6) {$B(v)$};

\draw (185,43) circle (12);
\draw (188,15) circle (12);

\node [w] at (199,58) {};
\node [w] at (201,60) {};
\foreach \x in {205,207}
	\node [w] at (\x,15) {};
\draw (170,29) -- +(5.9,5.9);
\draw (172,16) -- (175.5,16);
\draw (193,52) -- (197,56);
\draw (200,15) -- (203,15);

\node at (185,35) {$B(z^{w_1}_1)$};
\node at (188,7) {$B(z^{w_2}_1)$};

\foreach \x in {-5,-2}
	\node [w] at (\x,60) {};

\draw (180,40) -- (190,40);
\draw (183,12) -- (193,12);

\foreach \x in {40,45}
	\node [v] at (180,\x) {};
\foreach \x in {40}
	\node [v] at (190,\x) {};
\foreach \x in {12,17}
	\node [v] at (183,\x) {};
\foreach \x in {12}
	\node [v] at (193,\x) {};

\end{tikzpicture}
\caption{Tree-decomposition of a rank-expansion in Figure~\ref{fig:expansion}. The vertex sets $B(z^v_i)$ and $B(p^v_i)$, defined in Proposition~\ref{prop:pivotminor}, are bags which decompose $H[\overline{e}]$ and $H[S_v]$, respectively. }\label{fig:td}
\end{figure}

For the third axiom, it suffices to show that for every $t\in V(H)$, $T'[\{z:B(z)\ni t\}]$ is a subtree of $T'$. Let $t=(u^e_j,e,v)\in V(H)$ for some $e=vw\in E(T)$ and $1\leq j\leq \abs{U_e}$. If $v$ is the head of $e$, $T'[\{z:B(z)\ni t\}]=T'[\{z^v_j, \ldots, z^v_{\abs{U_e}}, p^v_1, \ldots, p^v_j\}]$, and it forms a path. Suppose $v$ is the tail of $e$. Let $f$ be the edge incoming to $v$, and if $a\in U_f$, then let $h$ be the integer such that $a=u^f_h$, if otherwise, let $h=1$. Then $T'[\{z:B(z)\ni t\}]=T'[\{p^v_h, \ldots, p^v_{\abs{U_e}}, v, z^w_1, \ldots, z^w_j\}]$. It also forms a path, thus $(T',\{B(v)\}_{v\in V(T')})$ is a tree-decomposition of $H$.

Since $\abs{B(y)}\leq 2k+1$ and for each $v\in V_I(T)\setminus \{y\}$ with the incoming edge $e$, $\abs{B(z^v_i)}=\abs{B(z^v_1)}=\abs{R_e}+1\leq k+1$, $\abs{B(p^v_i)}=\abs{B(p^v_1)}=\abs{R^v}+\abs{L_e}+1\leq (2k-\abs{U_e})+\abs{U_e}+1=2k+1$ and $\abs{B(v)}\leq 2k$, the resulting tree-decomposition has width at most $2k$.

Suppose that $G$ has linear rank-width at most $k$. Here, we choose $x\in V(T)$ such that $x$ is an end of a longest path in $T$, and let $y$ be the neighbor of $x$. For $v\in V_I(T)$ with outgoing edges $f_1$ and $f_2$, $\abs{U_{f_1}}=1$ or $\abs{U_{f_2}}=1$ because every inner vertex of $T$ is incident with a leaf. Therefore, for each $v\in V_I(T)\setminus \{y\}$ and $1\leq i\leq \abs{U_e}$, $\abs{B(p^v_i)}\leq (k+1-\abs{U_e})+\abs{U_e}+1=k+2$ and $\abs{B(v)}\leq k+1$, and $\abs{B(y)}\leq k+2$. Moreover, since $T[V_I(T)]$ is a path, $T'$ is also a path. Therefore $(T',\{B(v)\}_{v\in V(T')})$ is a path-decomposition of $H$ with path-width at most $k+1$.
\end{proof}

\begin{proof}[Proof of Theorem \ref{thm:main1}]
If $k=0$, then it is trivial. We assume that $k\geq 1$. We proceed by induction on the number of vertices.

Suppose $G$ is connected. Since $G$ has rank-width at most $k$ and $\abs{V(G)}\geq 3$, by Proposition~\ref{prop:treewidth}, there is a rank-expansion $H$ of $G$ such that $\tw(H)\leq 2k$, and $\abs{V(H)}\leq (2k+1)\abs{V(G)}-6k$. By Proposition~\ref{prop:pivotminor}, $H$ has a pivot-minor isomorphic to $G$.

If $G$ is disconnected, then we choose a largest component $Y$ of $G$. Since $k\geq 1$, the component $Y$ has at least $2$ vertices. If $\abs{V(Y)}=2$, then $G$ has rank-width $1$ and tree-width $1$, and $\abs{V(G)}\leq (2+1)\abs{V(G)}-6$ since $\abs{V(G)}\geq 3$. We assume that $\abs{V(Y)}\geq 3$. Then by induction hypothesis, there is a graph $H_1$ such that $Y$ is isomorphic to a pivot-minor of $H_1$ and $\tw(H_1)\leq 2k$ and $\abs{V(H_1)}\leq (2k+1)\abs{V(Y)}-6k$.

If $G\setminus V(Y)$ has tree-width at most $1$, then $G$ is isomorphic to a pivot-minor of the disjoint union of two graphs $H_1$ and $G\setminus V(Y)$, and the tree-width of it is equal to the tree-width of $H_1$. Since $\abs{V(H_1)}+\abs{V(G)\setminus V(Y)}\leq (2k+1)\abs{V(Y)}-6k+\abs{V(G)\setminus V(Y)}\leq (2k+1)\abs{V(G)}-6k$, we obtain the result. If tree-width of $G\setminus V(Y)$ is at least $2$, then $\abs{V(G)\setminus V(Y)}\geq 3$. Therefore, by induction hypothesis, there is a graph $H_2$ such that $G\setminus V(Y)$ is isomorphic to a pivot-minor of $H_2$ and $\tw(H_2)\leq 2k$ and $\abs{V(H_2)}\leq (2k+1)\abs{V(G)\setminus V(Y)}-6k$. So $G$ is isomorphic to a pivot-minor of the disjoint union of two graphs $H_1$ and $H_2$, and the tree-width of it is at most $2k$, and $\abs{V(H_1)}+\abs{V(H_2)}\leq (2k+1)\abs{V(G)}-6k$. Thus, we conclude the theorem. 
\end{proof}

\begin{proof}[Proof of Theorem \ref{thm:main2}]
We can easily obtain the proof of Theorem~\ref{thm:main2} from the proof of Theorem~\ref{thm:main1}.
\end{proof}

\section{Graphs with rank-width or linear rank-width at most $1$} 

Distance-hereditary graphs are introduced by Bandelt and Mulder~\cite{BM1986}. A graph $G$ is \emph{distance-hereditary} if for every connected induced subgraph $H$ of $G$ and vertices $a$, $b$ in $H$, the distance between $a$ and $b$ in $H$ is the same as in $G$. Oum~\cite{Oum05} showed that distance-hereidtary graphs are exactly graphs of rank-width at most $1$. Recently, Ganian~\cite{Ganian10} obtain a similar characterization of graphs of linear rank-width $1$. In this section, we obtain another characterizations for these classes in terms of vertex-minor relation.

Note that every tree has rank-width at most $1$ and every path has linear rank-width at most $1$.

\begin{THM}\label{thm:rwd1}
Let $G$ be a graph. The following are equivalent:
\begin{enumerate}
\item $G$ has rank-width at most $1$.
\item $G$ is distance-hereditary.
\item $G$ has no vertex-minor isomorphic to $C_5$.
\item $G$ is a vertex-minor of a tree.
\end{enumerate}
\end{THM}

\begin{proof}
$((1)\Leftrightarrow (2))$ is proved by Oum~\cite{Oum05}, and $((2)\Leftrightarrow (3))$ follows from the Bouchet's theorem~\cite{Bouchet87,Bouchet88}. Since every tree has rank-width at most $1$, $((4)\Rightarrow (1))$ is trivial. We want to prove that $(1)$ implies $(4)$.

Let $G$ be a graph of rank-width at most $1$. We may assume that $G$ is connected. If $\abs{V(G)}\leq 2$, then $G$ itself is a tree. So we may assume that $\abs{V(G)}\geq 3$. Let $(T,L)$ be a rank-decomposition of $G$ of width $1$. From Proposition~\ref{prop:pivotminor}, a rank-expansion $H$ with the rank-decomposition $(T,L)$ has $G$ as a pivot-minor.

The width of each edge in $T$ is $1$. Thus for $v\in V_I(T)$, the subgraph $H[S_v]$ is a path of length $2$ or a triangle because $G$ is connected. Also for $e\in E_I(T)$, $H[\overline{e}]$ consists of an edge. Therefore $H$ is connected and does not have cycles of length at least $4$.

Let $Q$ be a tree obtained from $H$ by replacing each triangle $abc$ with $K_{1,3}$ by adding a new vertex $d$, making $d$ adjacent to $a$, $b$, $c$ and deleting $ab$, $bc$, $ca$. Clearly $H$ is a vertex-minor of the tree $Q$ because we can obtain the graph $H$ from $Q$ by applying local complementation on those new vertices and deleting them. Therefore $G$ is a vertex-minor of a tree, as required.
\end{proof}

\begin{figure}[t]
\setlength{\unitlength}{.020in}
\begin{picture}(55,55)

\put(10,30){\circle*{3}}
\put(20,10){\circle*{3}}
\put(30,50){\circle*{3}}
\put(40,10){\circle*{3}}
\put(50,30){\circle*{3}}

\put(10,30){\line(1,1){20}}
\put(10,30){\line(1,-2){10}}
\put(20,10){\line(1,0){20}}
\put(40,10){\line(1,2){10}}
\put(50,30){\line(-1,1){20}}

\put(30,2){\makebox(0,0){$C_5$}}
\end{picture}
\begin{picture}(60,60)

\put(12,12){\circle*{3}}
\put(48,12){\circle*{3}}
\put(30,50){\circle*{3}}

\put(24,24){\circle*{3}}
\put(30,36){\circle*{3}}
\put(36,24){\circle*{3}}

\put(12,12){\line(1,1){12}}
\put(48,12){\line(-1,1){12}}
\put(30,50){\line(0,-1){14}}
\put(24,24){\line(1,2){6}}
\put(30,36){\line(1,-2){6}}
\put(36,24){\line(-1,0){12}}

\put(30,2){\makebox(0,0){$N$}}
\end{picture}
\begin{picture}(60,60)

\put(10,30){\circle*{3}}
\put(22,30){\circle*{3}}
\put(38,30){\circle*{3}}
\put(30,38){\circle*{3}}
\put(30,22){\circle*{3}}
\put(50,30){\circle*{3}}

\put(10,30){\line(1,0){12}}
\put(22,30){\line(1,1){8}}
\put(30,38){\line(1,-1){8}}
\put(38,30){\line(-1,-1){8}}
\put(30,22){\line(-1,1){8}}
\put(38,30){\line(1,0){12}}

\put(30,2){\makebox(0,0){$Q$}}

\end{picture}
\caption{The graphs $C_5$, $N$ and $Q$.}\label{fig:lrwd}
\end{figure}
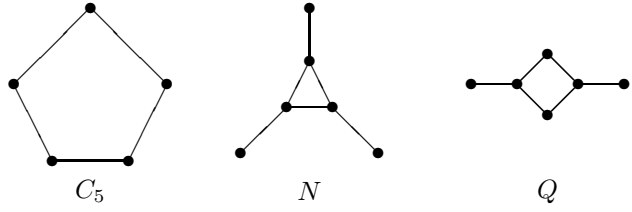

We also obtain a characterization of graphs with linear rank-width at most $1$.  Obstructions sets for graphs of linear rank-width $1$ are $C_5$, $N$ and $Q$~\cite{Adler11}, depicted in Figure~\ref{fig:lrwd}.

\begin{LEM}\label{lem:caterpillar}
Every subcubic caterpillar is a pivot-minor of a path.
\end{LEM}
\begin{proof}
Let $H$ be a subcubic caterpillar. By the definition of a caterpillar, there is a path $P$ in $H$ such that every vertex in $V(H)\setminus V(P)$ is a leaf. We choose such path $P=p_1p_2 \ldots p_m$ in $H$ with maximum length. We construct a path $Q$ from $P$ by replacing each edge $p_i p_{i+1}$ with a path $p_ia_ib_ip_{i+1}$. We can obtain a pivot-minor of $Q$ isomorphic to $P$ by pivoting each edge $a_ib_i$ and deleting all $a_i$ and deleting $b_i$ if $p_i$ is not adjacent to a leaf in $H$.
\end{proof}

\begin{THM}\label{thm:lrwd1}
Let $G$ be a graph. The following are equivalent:
\begin{enumerate}
\item $G$ has linear rank-width at most $1$.
\item $G$ has no vertex-minor isomorphic to $C_5$, $N$ or $Q$.
\item $G$ is a vertex-minor of a path.
\end{enumerate}
\end{THM}

\begin{proof}
$((1)\Leftrightarrow (2))$ is proved by Adler, Farley and Proskurowski~\cite{Adler11}. Since every path has linear rank-width at most $1$, $((3)\Rightarrow (1))$ is trivial. Let us prove that $(1)$ implies~$(3)$.

Let $G$ be a graph of linear rank-width at most 1. We may assume that $G$ is connected and $\abs{V(G)}\geq 3$. Let $H$ be a rank-expansion of $G$ with a linear rank-decompostion $(T,L)$ of width $1$. Note that $T$ is a caterpillar.

Since $(T,L)$ is a linear rank-decomposition of width $1$, for each triangle in $H$, one of those vertices is of degree $2$ in $H$. Let $P$ be a subcubic caterpillar obtained from $H$ by replacing each triangle with a path of length $2$ whose internal vertex has degree $2$ in $H$. We can obtain $H$ from $P$ by applying local complementation on the inner vertex of those paths of length $2$, $H$ is a vertex-minor of $P$. And by Lemma~\ref{lem:caterpillar}, $P$ is a pivot-minor of a path. Therefore $G$ is a vertex-minor of a path.
\end{proof}

\begin{figure}[t]
\setlength{\unitlength}{.032in}

\begin{picture}(78,35)

\multiput(10,17)(-2,0){3}{\circle*{1}}
\multiput(70,17)(2,0){3}{\circle*{1}}

\put(15,10){\circle*{2}}
\put(35,10){\circle*{2}}
\put(45,10){\circle*{2}}
\put(65,10){\circle*{2}}

\put(25,25){\circle*{2}}
\put(55,25){\circle*{2}}

\put(40,0){\makebox(0,0){$H$}}
\put(15,5){\makebox(0,0){$a$}}
\put(25,30){\makebox(0,0){$u$}}

\put(35,5){\makebox(0,0){$b$}}
\put(45,5){\makebox(0,0){$c$}}
\put(55,30){\makebox(0,0){$v$}}

\put(65,5){\makebox(0,0){$d$}}

\put(5,10){\line(1,0){10}}
\put(15,10){\line(1,0){20}}
\put(35,10){\line(1,0){10}}
\put(45,10){\line(1,0){20}}
\put(65,10){\line(1,0){10}}
\put(45,10){\line(2,3){10}}
\put(15,10){\line(2,3){10}}
\put(25,25){\line(2,-3){10}}

\end{picture}
\begin{picture}(70,35)

\put(40,0){\makebox(0,0){$P$}}
\multiput(10,17)(-2,0){3}{\circle*{1}}
\multiput(70,17)(2,0){3}{\circle*{1}}
\put(15,10){\circle*{2}}
\put(35,10){\circle*{2}}
\put(45,10){\circle*{2}}
\put(65,10){\circle*{2}}

\put(25,25){\circle*{2}}

\put(51,10){\circle*{1.5}}
\put(58,10){\circle*{1.5}}

\put(51,15){\makebox(0,0){$x$}}
\put(58,15){\makebox(0,0){$v$}}

\put(15,5){\makebox(0,0){$a$}}
\put(25,30){\makebox(0,0){$u$}}
\put(35,5){\makebox(0,0){$b$}}
\put(45,5){\makebox(0,0){$c$}}
\put(65,5){\makebox(0,0){$d$}}

\put(5,10){\line(1,0){10}}
\put(35,10){\line(1,0){10}}
\put(45,10){\line(1,0){20}}
\put(65,10){\line(1,0){10}}
\put(15,10){\line(2,3){10}}
\put(25,25){\line(2,-3){10}}
\end{picture}

\caption{A rank-expansion $H$ of a graph with linear rank-width $1$. The graph $H$ can be obtained from a path $P$ by applying local complementation on $u$ and pivoting $xv$ and deleting $x$.}
\label{fig:cat}
\end{figure}
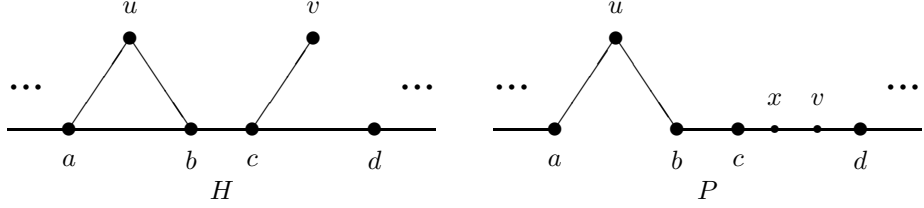

In Theorems \ref{thm:rwd1} and \ref{lem:caterpillar}, if a given graph is bipartite, we do not need to apply local complementation at some vertices. To prove it, we need the following lemma.

\begin{LEM}\label{lem:bipartite}
Let $G$ be a connected bipartite graph with rank-width $1$ and $\abs{V(G)}\geq 3$. Let $(T,L)$ be a rank-decomposition of width $1$. Then a rank-expansion of $G$ with respect to $(T,L)$ is a tree.
\end{LEM}

\begin{proof}
Let $x\in V(T)$ be a leaf and $H$ be a rank-expansion $\boldsymbol{R}(G,T,L,x,\{U_f\}_{f\in E(T)})$ of $G$.

Suppose that $H$ has a triangle. Then there exists a vertex $v\in V_I(T)$ such that $H[S_v]$ is the triangle. Let  $e_1$, $e_2$ and $e_3$ be edges incident with $v$ and assume that $e_1$ is the incoming edge. Let $U_{e_1}=\{a\}$, $U_{e_2}=\{b\}$ and $U_{e_3}=\{c\}$. By the construction of a rank-expansion, $bc\in E(G)$ and $R^{e_1}_a=R^{e_1}_b=R^{e_1}_c$. Since $R^{e_1}_a$ is a non-zero vector, there is a vertex $x\in V(G)$ such that $x$ is adjacent to all of $a$, 
$b$ and $c$. Therefore $xbc$ is a triangle in $G$, contradiction.
\end{proof}

\begin{THM}
Let $G$ be a graph. Then $G$ is bipartite and has rank-width at most $1$ if and only if $G$ is a pivot-minor of a tree.
\end{THM}

\begin{proof}
We may assume that $G$ is connected. Since every tree has rank-width at most $1$, backward direction is trivial. If $G$ is bipartite and has rank-width at most $1$, then by Lemma~\ref{lem:bipartite}, we have a rank-expansion of $G$ which is a tree. Hence, $G$ is a pivot-minor of a tree. 
\end{proof}

\begin{THM}
Let $G$ be a graph. Then $G$ is bipartite and has linear rank-width $1$ if and only if $G$ is a pivot-minor of a path.
\end{THM}

\begin{proof}
We may assume that $G$ is connected. Similarly, backward direction is trivial. Suppose $G$ is bipartite and has linear rank-width $1$. Let $H$ be a rank-expansion of $G$ with a linear rank-decomposition $(T,L)$ of width $1$. By Lemma~\ref{lem:bipartite}, the graph $H$ is a tree, and since $T$ is a subcubic caterpillar, $H$ is also a subcubic caterpillar. By Lemma~\ref{lem:caterpillar}, $H$ is a pivot-minor of a path, and so is $G$. 
\end{proof}


\end{document}